\documentclass[10pt,reqno]{amsart}
\usepackage{geometry}
\usepackage[utf8]{inputenc}
\usepackage{amsthm,amsfonts,amssymb,amstext}
\usepackage{enumerate}

\usepackage{graphicx}
\usepackage{caption}
\usepackage{subcaption}

\usepackage{hyperref} 
\hypersetup{
  colorlinks,%
  citecolor=blue,%
    filecolor=red,%
    linkcolor=red,%
    urlcolor=blue
}

\numberwithin{equation}{section}
\newtheorem{thm}[equation]{Theorem}
\newtheorem{lem}[equation]{Lemma}
\newtheorem{prop}[equation]{Proposition}
\newtheorem{defn}[equation]{Definition}

\newcommand{\Rn}{\mathbb{R}^{n}}
\newcommand{\R}{\mathbb{R}}

\newcommand{\abs}[1]{\left\vert#1\right\vert}

\begin{document}

\title[]{Large-scale homogeneity and isotropy versus fine-scale condensation. A model based on Muckenhoupt type densities}

\author[]{Hugo Aimar}

\author[]{Federico Morana}



\begin{abstract}
In this brief note we aim to provide, through a well known class of singular densities in harmonic analysis, a simple approach to the fact that the homogeneity of the universe on scales of the order of a hundred millions light years is completely compatible with the fine-scale condensation of matter and energy. We give precise and quantitative definitions of homogeneity and isotropy on large scales. Then we show that Muckenhoupt densities have the ingredients required to a model for the large-scale homogeneity and the fine-scale condensation of the universe. In particular, these densities can take locally infinitely large values (black holes) and at once, in the large scales they are independent of the location. We also show some locally singular densities satisfying the large scale isotropy property.
\end{abstract}

\maketitle

\section[intro]{Introduction}\label{S.1}
The cosmological principle of the universe states that the universe is homogeneous and isotropic at large scales. Large means here lengths larger than the diameter of galaxy clusters (see Chapter~27 of \cite{MTWgravitation}). But, of course, no homogeneity could be expected at small scales. Planets, stars, galaxies are real inhomogeneities in which there is concentration of mass and energy.

\medskip

From the point of view of mathematics, homogeneity has also been considered from several points of view. Perhaps the most simple but non-quantitative meaning is that of topological homogeneity. In this approach the expression ``the space looks the same no matter where you are'' is interpreted from the topological perspective. Hence a topological space $X$ is homogeneous if for every couple of points $x$ and $y$ in $X$ there exists a homeomorphism $f:X\rightarrow X$ such that $f(x)=y$.  This means that
the universe looks topologically the same about $x$ than about $y$. Nevertheless this concept does not take into account the local heterogeneities and it is not quantitative.

\medskip

For simplicity, assume that the ``substance'' of an $n$-dimensional Euclidean universe is distributed according to a density $\rho(x)$. Here $x=(x_1,\cdots,x_n)$ is a point of $\R^n$ and $\rho(x)\geq 0$.

\medskip

Let us proceed to give a plausible definition of homogeneity in the large scales for a general density $\rho(x)$ of our $n$-dimensional space.  Assume that our universe has two astronomers, located at different points $x^1$ and $x^2$.  Assume that their astronomical skills grow with the same rhythm. In particular they are able to weigh their own neighborhoods. Set $B(x^i, R),\ i=1,2$, to denote the euclidean ball centered at $x^i$ with radius $R>0$. Precisely $B(x^i,R)=\{x\in\R^n:|x-x^i|<R\}$. At the beginning their observations will be disjoint, when $R$ increases they will start sharing part of their observable universe. When $R$ is very large, they will share most of their measurements but the two observed regions will never coincide.

\medskip

\begin{figure}[!tbp]
	\begin{subfigure}[b]{0.33\textwidth}
		\includegraphics[width=\textwidth, height=\textwidth]{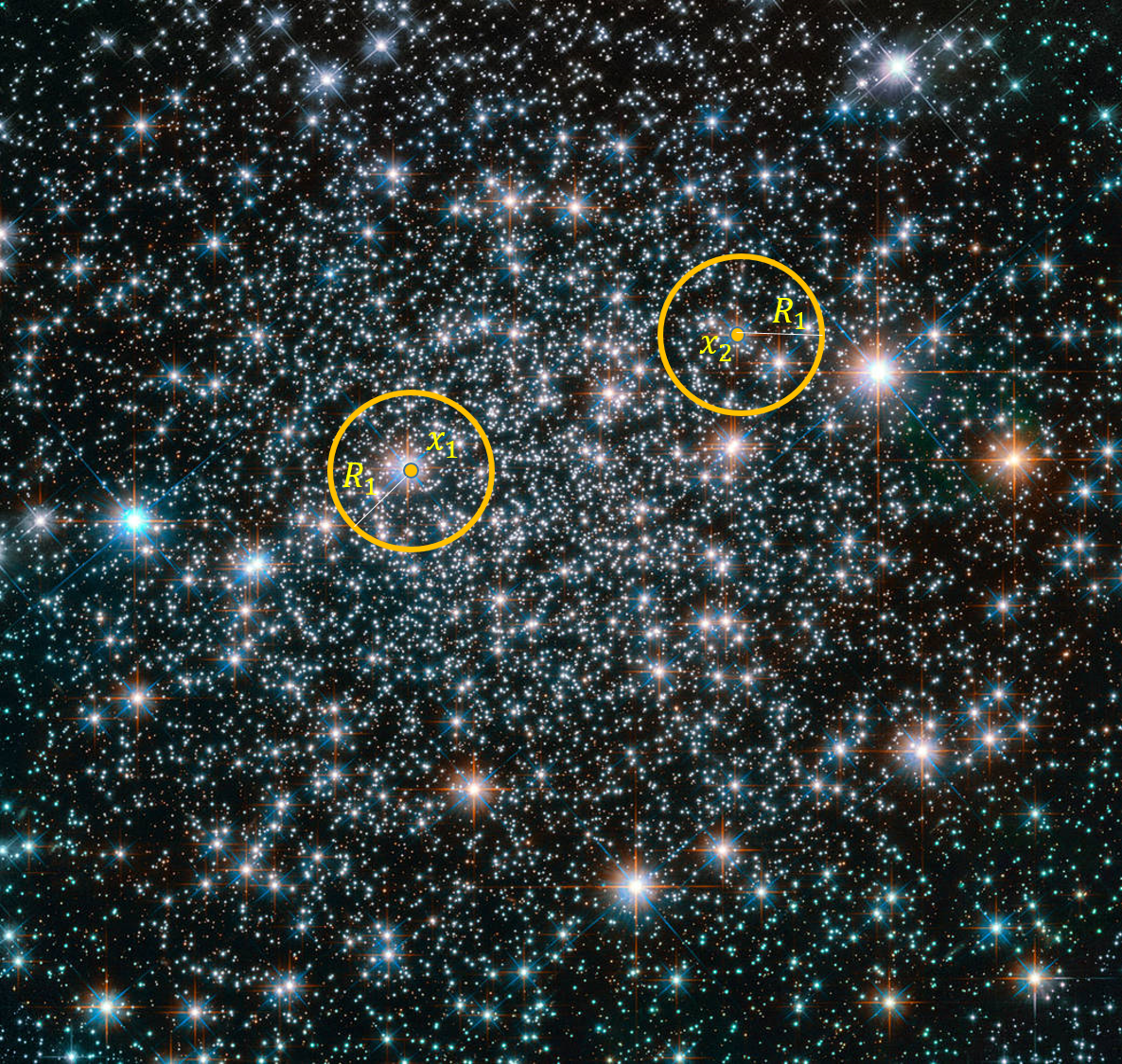}
		\label{fig:f1}
	\end{subfigure}
	\begin{subfigure}[b]{0.33\textwidth}
		\includegraphics[width=\textwidth, height=\textwidth]{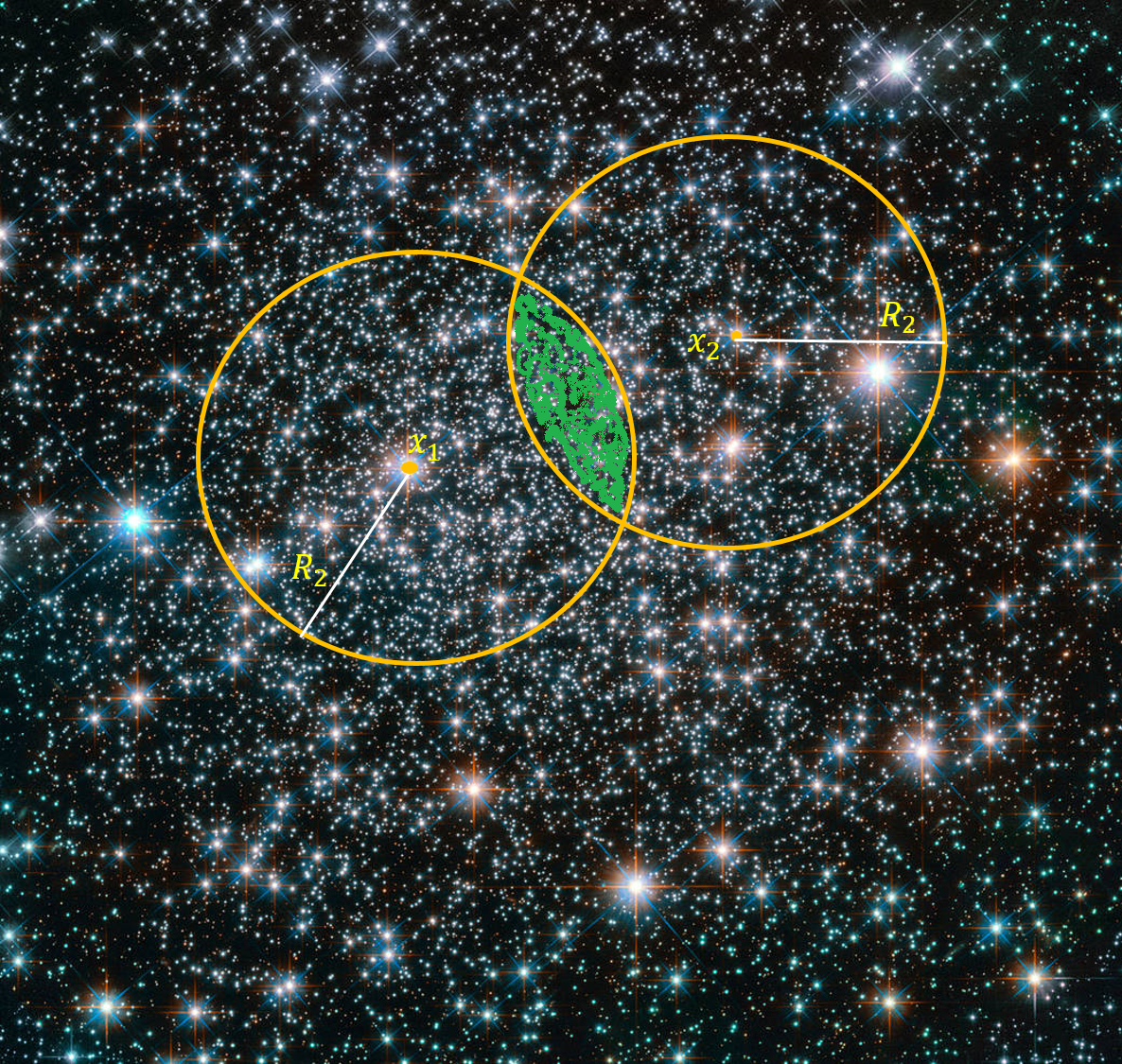}
		\label{fig:f2}
	\end{subfigure}
	\begin{subfigure}[b]{0.33\textwidth}
		\includegraphics[width=\textwidth, height=\textwidth]{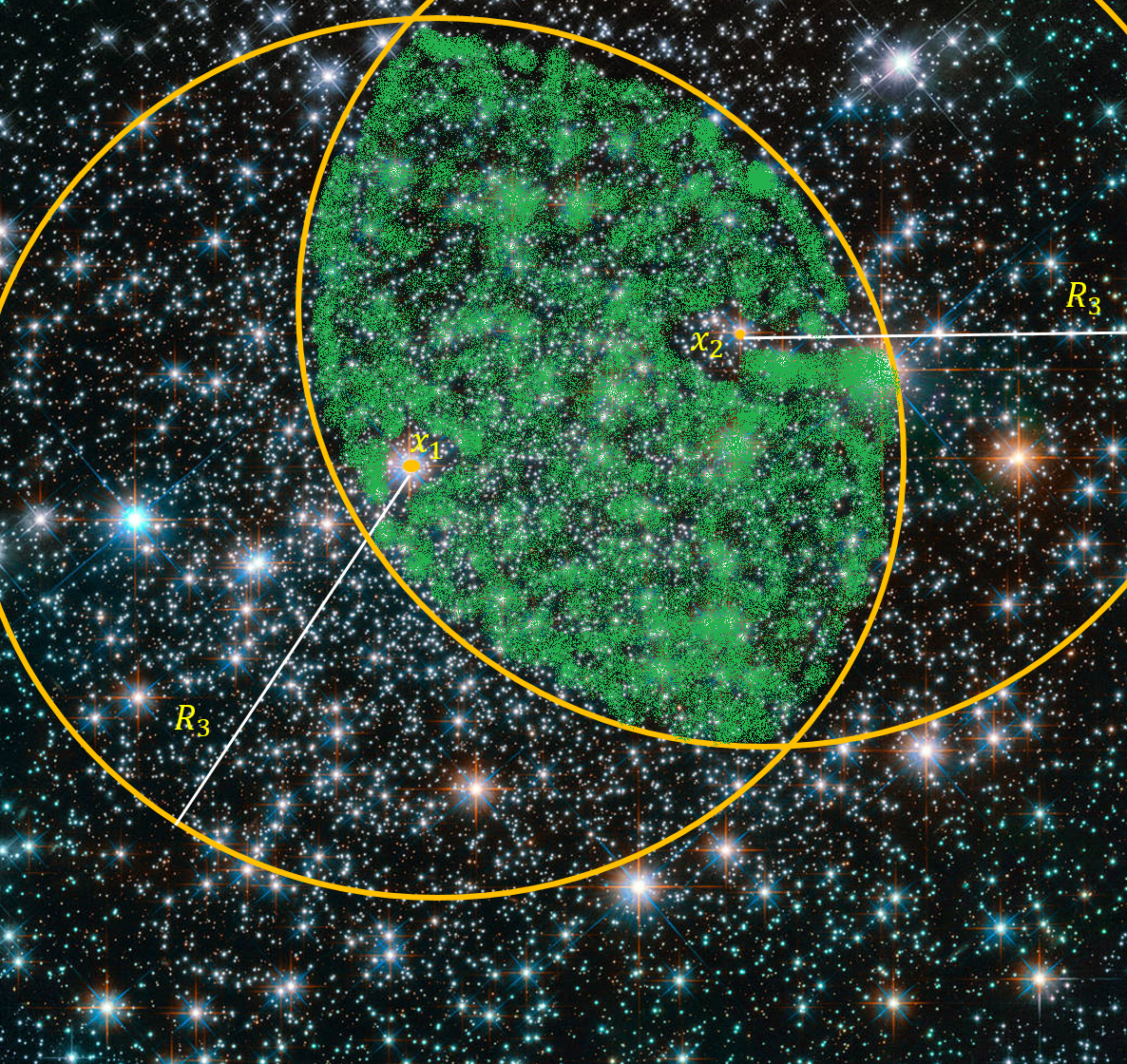}
		\label{fig:f3}
	\end{subfigure}
	\caption{Two observers at $x^1$ and $x^2$ improving their astronomical skills. (NASA Goddard. \href{https://images.nasa.gov/details-GSFC_20171208_Archive_e000282}{https://images.nasa.gov/details-GSFC\_20171208\_Archive\_e000282})}
\end{figure}

In our elementary interpretation, the two observers are drawing the increasing functions

$$ M_1(R)=\int_{B(x^1,R)}\rho(x)\, dx $$

\noindent and

$$ M_2(R)=\int_{B(x^2,R)}\rho(x)\, dx $$

\noindent respectively. We say that the universe, modeled by $\rho(x)$, is homogeneous in the large scales if no matter what the locations $x^1$ and $x^2$ are, the two observers are weighing the same mass for $R$ large. More precisely if
\begin{equation}\label{homog.1}
    \frac{M_1(R)}{M_2(R)} \xrightarrow[R\to\infty]{}1
\end{equation}
for every choice $x^1$, $x^2$ of points in $\Rn$.
It is worthy noticing that not every locally integrable density satisfies this property. In fact, as it is easy to prove, exponential growth is not allowed.

\medskip
The question is whether or not there are nontrivial densities satisfying \eqref{homog.1}.  Non triviality should allow very strong local heterogeneities. Not only objects of large mass but also points of infinitely large density in order to include, at least in an heuristic way, an elementary model for black holes.

\medskip

We aim to show that there is a huge class of densities $\rho$ in $\R^n$ fulfilling this two apparently contradictory conditions: homogeneity in the large scales, in the sense of \eqref{homog.1}, and strong local heterogeneity allowing heavy objects and locally unbounded density.

\medskip

The class of densities that we shall consider is the well known, in harmonic analysis, class of Muckenhoupt weights. See \cite{Muckenhoupt72}, \cite{GarciaCuervaRubioFranciaBook}, \cite{CoifmanFefferman74}.

\medskip

Section~\ref{S.2} is devoted to a brief introduction of this class of densities and to exhibit some nontrivial examples, including densities that are singular on sets of positive dimension. In Section~\ref{S.3} we prove that the densities defined by Muckenhoupt weights provide the desired homogeneity at large scales in the sense of \eqref{homog.1}. In Section~\ref{S.4} we briefly discuss the corresponding isotropy condition with these types of mass distributions.

\section[MuckDens]{Muckenhoupt densities}\label{S.2}

Let us start by introducing the Muckenhoupt classes in the Euclidean setting $\Rn$.  There are extensive treatments of the subject in harmonic analysis such as \cite{Muckenhoupt72}, \cite{GarciaCuervaRubioFranciaBook}, \cite{CoifmanFefferman74}. These approaches focus usually more on operator theory, than into the  geometrical properties associated to these densities. With $B$ we shall denote the Euclidean open balls of $\Rn$.  We shall write $B(x,r)$ when the center $x\in\Rn$ and the radius $r>0$ of $B$ have to be explicit. A positive measurable and locally integrable function $\rho$ defined in $\Rn$ satisfies always, by Schwarz inequality, the following upper estimate for the volume $|B|$ of a ball $B$
\begin{align*}
    |B|^2 & = \left(\int_B\,dx\right)^2=
    \left(\int_B\rho^{\frac{1}{2}}\rho^{-\frac{1}{2}}\,dx\right)^2\\
    & \leq \left(\int_B\rho\,dx\right)\left(\int_B\rho^{-1}\,dx\right).
\end{align*}
for every ball $B$. More generally, from H\"older inequality with $1<p<\infty$ and $\frac{1}{q}+\frac{1}{p}=1$, we have also
\begin{align*}
|B|^p & = \left(\int_B\, dx\right)^p = \left(\int_B\rho^{\frac{1}{p}}\rho^{-\frac{1}{p}}\, dx\right)^p\\
& \leq \left\{\left(\int_B\rho\,
dx\right)^{\frac{1}{p}}\left(\int_B\rho^{-{\frac{q}{p}}}\,
dx\right)^{\frac{1}{q}}\right\}^p\\
& = \left(\int_B\rho\, dx\right)\left(\int_B\rho^{-\frac{1}{p-1}}\,
dx\right)^{p-1}
\end{align*}
again for every ball $B$ and every density $\rho$.

\medskip

It is easy to provide examples of densities in which the above estimates for the volume of $B$ are far away from being optimal. On the other hand it is clear that the above inequalities become equations for balls in regions where the density $\rho$ is constant.
And that $|B|^p$ and $\left(\int_B\rho\,dx\right) \left(\int_B\rho^{-\frac{1}{p-1}}\, dx\right)^{p-1}$ are of the same order if there exist positive and finite constants $0<c_1\leq \rho(x)\leq c_2<\infty$ for every $x\in\Rn$. In other words, these two magnitudes become equivalent when $\rho$ is almost constant. On the other hand, from the point of view of the model for a density of the universe, the property $0<c_1\leq \rho(x)\leq c_2<\infty$ is unrealistic because it prevents having very localized and massive objects and even for having some points with vanishing density. Precisely, Muckenhoupt densities deserve analysis and particular consideration because the quantities $|B|^p$ and $\left(\int_B\rho\,dx\right)\left(\int_B\rho^{-\frac{1}{p-1}}\,dx\right)^{p-1}$ are equivalent and $\rho$ can still have singularities, points of infinity density and points of vanishing density.

\smallskip

\begin{defn}
Let $\rho$ be a density in $\Rn$ and $1<p<\infty$. We shall say that $\rho$ is a Muckenhoupt density, of the class $A_p$, if there exists a constant $C$ such that
$$ \left(\int_B\rho\,dx\right)\left(\int_B\rho^{-\frac{1}{p-1}}\,dx\right)^{p-1}\leq C|B|^p $$
for every ball $B$ in $\Rn$.

A density $\rho$ is said to belong to $A_1$ if there exists a constant $C$ such that
\begin{equation*}
	\int_B \rho dx \leq C\inf_B \rho\ . \abs{B}
\end{equation*}
for every ball $B$ in $\mathbb{R}^n$.

A density $\rho$ is said to belong to $A_{\infty}$ if $\rho\in A_p$ for some $1\leq p<\infty$.
\end{defn}

\smallskip

Benjamin Muckenhoupt introduced in \cite{Muckenhoupt72} these densities as the exact classes for the boundedness of Hardy-Littlewood and singular integrals operators in Lebesgue spaces. In this brief introduction we shall only collect the measure theoretical aspects of Muckenhoupt densities allowing us to prove, in the next section, the homogeneity in the large scales of a
universe modeled by this type of densities $\rho$.

\medskip

The next proposition contains the mains properties of Muckenhoupt densities related with the subsequent development of our model.

\begin{prop}\label{P2.1}
Let $\rho$ be a density in $\Rn$ the $n$-dimensional Euclidean space. Then
\begin{description}
  \item[(a)] for $1<p<q$ and $\rho\in A_p$ we have that $\rho\in A_q$;
  \item[(b)] for $p>1$ and $\rho\in A_p$ there exists a constant $C$ such that the inequality
  $$ \frac{|E|}{|B|}\leq C\left(\frac{\rho(E)}{\rho(B)}\right)^{1/p}, $$
  $\rho(A)=\int_A\rho\,dx$, holds for every ball $B$ and every measurable subset $E$ of $B$;
  \item[(c)] for $p>1,\, \rho\in A_p$ and $C$ as in (b) the mass satisfies the following doubling property
  $$ \rho(Bx,2R))\leq 2^{np}C^p\rho(B(x,R)) $$
  for every $x\in \Rn$ and every $R>0$;
  \item[(d)] let $\rho$ be any $A_{\infty}$ density, then there exist two constants $\widetilde{C}>0$ and $\gamma>0$ such that the inequality
  $$ \frac{\rho(E)}{\rho(B)}\leq \widetilde{C}\left(\frac{|E|}{|B|}\right)^{\gamma}, $$

\noindent holds for every ball $B$ and every measurable subset $E$ of $B$.

\end{description}
\end{prop}

The proof of the above proposition can be found in \cite{CoifmanFefferman74} or \cite{GarciaCuervaRubioFranciaBook}.
The first three properties are simple consequences of the definition of $A_p$.
Property (d) instead relies on subtle arguments such as the Calder\'on-Zygmund decomposition technique.  On the other hand Property (d) shall be crucial in our proof of the large scale of homogeneity in the next section.

\smallskip

Let us give some examples of locally quite heterogeneous Muckenhoupt densities.  The first class of non trivial $A_{\infty}$ densities which are radial with respect to some point $x_0\in\Rn$ are the powers of the distance.  In fact $\rho(x)=|x-x_0|^{\beta}$ belongs to the Muckenhoupt class $A_p(\Rn)$ if and only if $-n<\beta<n(p-1)$.  Hence $\rho(x)=|x-x_0|^{\beta}$ belongs to $A_{\infty}$ if and only if $\beta>-n$.  When $\beta=0$ we recover the trivial constant weights but, when $-n<\beta<0$ we have a singularity in $x_0$ and $\rho$ can be arbitrarily large in any neighborhood of $x_0$.  On the other hand, when $\beta>0$, we have
densities that are unbounded at infinity, but they are still $A_{\infty}$ densities.  An extension of the above is due to Ricci and Stein \cite{RicciStein87} (see also \cite{SteinBook93}).

\smallskip

The singularities of the density can actually be concentrated on sets which are quite irregular. The most recent and interesting result concerning the structure of these sets is given in \cite{Anderson2022weakly}. The authors provide a necessary an sufficient condition on a set $F$ of the Euclidean space in order to produce $A_1$ Muckenhoupt weights as negative powers of $d(x,F)$. This condition is a weak form of porosity and essentially contains the situations described in \cite{DuLoG10}, \cite{AiCaDuTo14} and \cite{DyIhLeTuVa19}. The results are of the following type. The function $\rho(x)=d^{-\beta}(x,F)$ belongs to the class $A_1$ if and only if $F$ is (weakly) porous. In particular the function 
$$ \rho(x)=\frac{1}{\sqrt{d(x,F)}}$$
belongs to $A_\infty$ if $F$ is the smooth boundary of some open domain in $\mathbb{R}^n$.
Hence we can have densities with infinitely large values on very large sets of Hausdorff dimension less than $n$, for example on a surface of $\R^3$ .

\smallskip

Several important results of harmonic analysis provide tools for building singular weights. In particular the relation of Muckenhoupt classes with the space of functions of bounded mean oscillation.

\smallskip

In some sense these results revels that power laws are in some sense the extremals of the $A_{\infty}$ class of densities. No exponential behavior for a density, as illustrated in Section~\ref{S.1}, will produce the desired homogeneity in the large scales allowing strong heterogeneities in any bounded region of the space.

\section[TMainRes]{The main result}\label{S.3}

Let us formally state and prove the main result of this note.

\begin{thm}
Let $\rho(x)$ be an $A_{\infty}(\Rn)$ Muckenhoupt density. Then $\Rn$ with the density $\rho(x)$ is homogeneous at the large scales.
\end{thm}

\begin{proof}
Let $x^1$ and $x^2$ be two points in $\Rn$ and let $R>0$ be given. With the notation introduced in Section~\ref{S.2} we have to prove that
$$\frac{\rho(B(x^1,R))}{\rho(B(x^2,R))}\xrightarrow[R\to\infty]{}1\,.$$
With the standard notation for intersections and differences of sets we may write

\begin{align*}
    \frac{\rho(B(x^1,R))}{\rho(B(x^2,R))} & = \frac{\rho(B(x^1,R)\cap B(x^2,R))+\rho(B(x^1,R)\setminus B(x^2,R))}{\rho(B(x^2,R)\cap B(x^1,R))+\rho(B(x^2,R)\setminus B(x^1,R))} \\
    & = \dfrac{1+\dfrac{\rho(A_1)}{\rho(B(x^1,R)\cap
    B(x^2,R))}}{1+\dfrac{\rho(A_2)}{\rho(B(x^2,R)\cap B(x^1,R))}}
\end{align*}

\noindent with $A_1=B(x^1,R)\setminus B(x^2,R)$ and
$A_2=B(x^2,R)\setminus B(x^1,R)$. Notice that

\begin{enumerate}[(1)]
  \item $A_1\subset E_1= B(x^1,R)\setminus B(x^1,R-|x^1-x^2|)\subset B(x^1,R)$;
  \item $A_2\subset E_2= B(x^2,R)\setminus B(x^2,R-|x^1-x^2|)\subset B(x^2,R)$;
  \item $B(x^1,\frac{R}{2})\subset B(x^1,R)\cap B(x^2,R)$ for $R$ large enough;
  \item $B(x^2,\frac{R}{2})\subset B(x^1,R)\cap B(x^2,R)$ for $R$ large enough.
\end{enumerate}

Now use (3), (4) and the doubling property (c) in Proposition~\ref{P2.1} and obtain

\begin{equation*}
    \frac{\rho(A_i)}{\rho(B(x^1,R)\cap B(x^2,R))} \leq
    \frac{\rho(A_i)}{\rho(B(x^i,R/2))}
    \leq 2^{np}C^p \frac{\rho(A_i)}{\rho(B(x^i,R))}
\end{equation*}
for $i=1,2$ and $p$ such that $\rho\in A_p(\Rn)$. The constant $C$ is that in (c) of Proposition~\ref{P2.1}. Let us now apply (1) and (2) above and then (d) in Proposition~\ref{P2.1}.  Doing so, since $E_i\subset B(x_i,R)\, (i=1,2)$ we get
\begin{align*}
    \frac{\rho(A_i)}{\rho(B(x^1,R)\cap B(x^2,R))}
    & \leq 2^{np}C^p\frac{\rho(E_i)}{\rho(B(x^i,R))} \\
    & \leq 2^{np}C^p\widetilde{C} \left(\frac{|E_i|}{|B(x^i,R)|}\right)^{\gamma}\\
    & = 2^{np}C^p\widetilde{C} \left[\frac{R^n-(R-|x^1-x^2|)^n}{R^n}\right]^{\gamma}\\
    & = 2^{np}C^p\widetilde{C}\left[1-\left(1-\frac{|x^1-x^2|}{R}\right)^n\right]^{\gamma}
\end{align*}
for $i=1,2$. Since the above expression tends to zero for $R\to \infty$, we have the desired homogeneity in the large scales, i.e.
$$\frac{\rho(B(x^1,R))}{\rho(B(x^2,R))}
\xrightarrow[R\to\infty]{}1\,.$$
Actually the last estimate provides a velocity of convergence in terms of the distance $|x^1-x^2|$ between the two observers.
\end{proof}

\section[Iso]{Isotropy}\label{S.4}

The second component of the cosmological principle is the large scale isotropy of the universe. Roughly speaking, isotropy entails that from every point of the space the universe looks, at large scales, the same in any direction.

Assume as before that the density of the universe of $n$ dimensions is distributed according to the nonnegative density function $\rho(x)=\rho(x_1,\dots,x_n)$. A quantitative version of the above qualitative formulation can be given by the fact that
\begin{equation}\label{iso.def}
	\lim\limits_{R\to\infty}\, \frac{\lambda(x_1,\vec{v_1},R)}{\lambda(x_2,\vec{v_2},R)} = 1
\end{equation}
for every $x_1,x_2 \in \Rn$ and $\vec{v_1},\vec{v_2}\in S^{n-1}$, where
\begin{equation*}
\lambda(x,\vec{v},R) = \int_{0}^{R} \rho(x+t\vec{v})\,dt    
\end{equation*}
for $x\in\Rn$ and $\vec{v}\in S^{n-1}$ the unit sphere in $\Rn$.

Notice that in order to have the function $\lambda$ well defined for every $x\in\Rn$, every $R>0$ and every $\vec{v}\in S^{n-1}$ we need more than the local integrability of $\rho$. Nevertheless some local singularities of $\rho$ are still possible with power laws weaker than those providing homogeneity. The next result shows a simple example of isotropy and homogeneity with a singularity.

\begin{lem}\label{lemma:functionlambdaaboveandbelow}
	Let $x_0$ be a fixed point in $\Rn$ and $0<\alpha<1$. Let $\rho(x)=\abs{x-x_0}^{-\alpha}$, then
	\begin{align*}
	 \frac{1}{1-\alpha} \left[ \big(\abs{x-x_0}+R\big)^{1-\alpha} - \abs{x-x_0}^{1-\alpha} \right]  &\leq \lambda(x,\vec{v},R) \\ 
	 &\leq \frac{\abs{x-x_0}^{1-\alpha}}{1-\alpha} + \frac{1}{1-\alpha} \big(R-\abs{x-x_0}\big)^{1-\alpha}   
	\end{align*}
	for every $x\in\Rn$, $\vec{v}\in S^{n-1}$ and $R>\abs{x-x_0}$.
\end{lem}
\begin{proof}
	Let us first estimate pointwise from above and from below the restriction of the density $\rho$ to the half line $\{x+t\vec{v}: t>0\}$. Since $-\abs{x-x_0}\leq (x-x_0)\cdot\vec{v}\leq \abs{x-x_0}$, we have, for $t>0$ that $-2t\abs{x-x_0}\leq 2t (x-x_0)\cdot\vec{v}\leq 2t\abs{x-x_0}$. Hence
	\begin{equation*}
	    \abs{(x-x_0) - t\frac{x-x_0}{\abs{x-x_0}}}^2\leq \abs{(x-x_0)+t\vec{v}}^2\leq \abs{(x-x_0) + t\frac{x-x_0}{\abs{x-x_0}}}^2.	
	 \end{equation*}
	 Then for $\rho(x+t\vec{v})=\abs{(x+t\vec{v})-x_0}^{-\alpha}$, we have
	 \begin{equation*}
	     \rho\left(x+t\frac{x-x_0}{\abs{x-x_0}}\right) \leq \rho(x+t\vec{v})\leq \rho\left(x-t\frac{x-x_0}{\abs{x-x_0}}\right).
	 \end{equation*}
	 So that, integrating for $t\in [0,R]$ we get
	 \begin{equation*}
	     \int_0^R \abs{(x-x_0) + t\frac{x-x_0}{\abs{x-x_0}}}^{-\alpha} dt \leq \lambda(x,\vec{v},R)\leq \int_0^R \abs{(x-x_0) - t\frac{x-x_0}{\abs{x-x_0}}}^{-\alpha} dt.
	 \end{equation*}
	 Then
	 \begin{equation*}
	     \int_0^R (t+\abs{x-x_0})^{-\alpha} dt \leq \lambda(x,\vec{v},R)\leq \int_0^R \abs{\abs{x-x_0} - t}^{-\alpha} dt,
	 \end{equation*}
	 or, for $R$ large
	 \begin{align*}
	     \frac{1}{1-\alpha}&\left[(\abs{x-x_0}+R)^{1-\alpha} - \abs{x-x_0}^{1-\alpha}\right]\\ &\leq \lambda(x,\vec{v},R)\\ &\leq \int_{0}^{\abs{x-x_0}} (\abs{x-x_0}-t)^{-\alpha} dt + \int_{\abs{x-x_0}}^R (t-\abs{x-x_0})^{-\alpha} dt\\
	     &= \frac{\abs{x-x_0}^{1-\alpha}}{1-\alpha} + \frac{1}{1-\alpha}(R-\abs{x-x_0})^{1-\alpha},
	 \end{align*}
	 as desired.
\end{proof}

\begin{prop}
For $x_0\in\Rn$ and $0<\alpha<1$, the density $\rho(x)=\abs{x-x_0}^{-\alpha}$ is homogeneous and isotropic.	
\end{prop}
\begin{proof}
Take $x_1$, $x_2\in \Rn$; $\vec{v_1}, \vec{v_2}\in S^{n-1}$ and $R$ large. Applying Lemma~\ref{lemma:functionlambdaaboveandbelow} we have
\begin{align*}
    \frac{\frac{1}{1-\alpha}\left[(\abs{x_1-x_0}+R)^{1-\alpha}-\abs{x_1-x_0}^{1-\alpha}\right]}{\dfrac{\abs{x_2-x_0}^{1-\alpha}}{1-\alpha}+ \dfrac{1}{1-\alpha}(R-\abs{x_2-x_0})^{1-\alpha}} &\leq\dfrac{\lambda(x_1,\vec{v_1},R)}{\lambda(x_2,\vec{v_2},R)}\\
    &\leq\frac{\dfrac{\abs{x_1-x_0}^{1-\alpha}}{1-\alpha}+ \dfrac{1}{1-\alpha}(R-\abs{x_1-x_0})^{1-\alpha}}{\frac{1}{1-\alpha}\left[(\abs{x_2-x_0}+R)^{1-\alpha}-\abs{x_2-x_0}^{1-\alpha}\right]}.
\end{align*}
The first and the last terms in the above inequalities tend to one for $R\to \infty$ and we are done, since $\rho$ is a Muckenhoupt weight.
\end{proof}

\medskip
\subsection*{Funding} 
This work was supported by the Ministerio de Ciencia, Tecnolog\'ia e Innovaci\'on-MINCYT in Argentina: Consejo Nacional de Investigaciones Cient\'ificas y T\'ecnicas-CONICET; and Universidad Nacional del Litoral-UNL.



\begin{thebibliography}{10}
	
	\bibitem{AiCaDuTo14}
	H.~Aimar, M.~Carena, R.~Dur\'{a}n, and M.~Toschi, \emph{Powers of distances to
		lower dimensional sets as {M}uckenhoupt weights}, Acta Math. Hungar.
	\textbf{143} (2014), no.~1, 119--137. \MR{3215609}
	
	\bibitem{Anderson2022weakly}
	Theresa~C. Anderson, Juha Lehrb\"{a}ck, Carlos Mudarra, and Antti~V.
	V\"{a}h\"{a}kangas, \emph{Weakly porous sets and {M}uckenhoupt {$A_p$}
		distance functions}, \href{https://doi.org/10.48550/arXiv.2209.06284}{
		arXiv:2209.06284} (2022).
	
	\bibitem{CoifmanFefferman74}
	R.~R. Coifman and C.~Fefferman, \emph{Weighted norm inequalities for maximal
		functions and singular integrals}, Studia Math. \textbf{51} (1974), 241--250.
	\MR{358205}
	
	\bibitem{DuLoG10}
	Ricardo~G. Dur\'{a}n and Fernando L\'{o}pez~Garc\'{\i}a, \emph{Solutions of the
		divergence and analysis of the {S}tokes equations in planar
		{H}\"{o}lder-{$\alpha$} domains}, Math. Models Methods Appl. Sci. \textbf{20}
	(2010), no.~1, 95--120. \MR{2606245}
	
	\bibitem{DyIhLeTuVa19}
	Bart\l~omiej Dyda, Lizaveta Ihnatsyeva, Juha Lehrb\"{a}ck, Heli Tuominen, and
	Antti~V. V\"{a}h\"{a}kangas, \emph{Muckenhoupt {$A_p$}-properties of distance
		functions and applications to {H}ardy-{S}obolev--type inequalities},
	Potential Anal. \textbf{50} (2019), no.~1, 83--105. \MR{3900847}
	
	\bibitem{GarciaCuervaRubioFranciaBook}
	Jos\'{e} Garc\'{\i}a-Cuerva and Jos\'{e}~L. Rubio~de Francia, \emph{Weighted
		norm inequalities and related topics}, North-Holland Mathematics Studies,
	vol. 116, North-Holland Publishing Co., Amsterdam, 1985, Notas de
	Matem\'{a}tica [Mathematical Notes], 104. \MR{807149}
	
	\bibitem{MTWgravitation}
	Charles~W. Misner, Kip~S. Thorne, and John~Archibald Wheeler,
	\emph{Gravitation}, W. H. Freeman and Co., San Francisco, Calif., 1973.
	\MR{0418833}
	
	\bibitem{Muckenhoupt72}
	Benjamin Muckenhoupt, \emph{Weighted norm inequalities for the {H}ardy maximal
		function}, Trans. Amer. Math. Soc. \textbf{165} (1972), 207--226. \MR{293384}
	
	\bibitem{RicciStein87}
	Fulvio Ricci and E.~M. Stein, \emph{Harmonic analysis on nilpotent groups and
		singular integrals. {I}. {O}scillatory integrals}, J. Funct. Anal.
	\textbf{73} (1987), no.~1, 179--194. \MR{890662}
	
	\bibitem{SteinBook93}
	Elias~M. Stein, \emph{Harmonic analysis: real-variable methods, orthogonality,
		and oscillatory integrals}, Princeton Mathematical Series, vol.~43, Princeton
	University Press, Princeton, NJ, 1993, With the assistance of Timothy S.
	Murphy, Monographs in Harmonic Analysis, III. \MR{1232192}
	
\end{thebibliography}

\providecommand{\bysame}{\leavevmode\hbox to3em{\hrulefill}\thinspace}
\providecommand{\MR}{\relax\ifhmode\unskip\space\fi MR }
\providecommand{\MRhref}[2]{%
	\href{http://www.ams.org/mathscinet-getitem?mr=#1}{#2}
}
\providecommand{\href}[2]{#2}


\medskip
\noindent{\footnotesize
\noindent\textit{Affiliation.\,}
\textsc{Instituto de Matem\'{a}tica Aplicada del Litoral, UNL, CONICET.}

\smallskip
\noindent\textit{Address.\,}\textmd{CCT CONICET Santa Fe, Predio ``Dr. Alberto Cassano'', Colectora Ruta Nac.~168 km 0, Paraje El Pozo, S3007ABA Santa Fe, Argentina.}

\smallskip
\noindent \textit{E-mail address.\, }Hugo Aimar (corresponding author), \verb|haimar@santafe-conicet.gov.ar|, 
Federico Morana, \verb|fmorana@santafe-conicet.gov.ar|
}

\end{document}